\documentclass[12pt]{amsart} 

\usepackage{amscd} 
\usepackage{amsmath} 
\usepackage{amssymb} 
\usepackage{amsthm}
\usepackage{bigdelim}
\usepackage{color} 
\usepackage{enumerate}
\usepackage{mathrsfs}
\usepackage{multirow}
\usepackage[all]{xy} 
\usepackage{hyperref}
\newtheorem{thm}{Theorem}[section]

\newtheorem{conj}[thm]{Conjecture}

\newtheorem{lem}[thm]{Lemma}
\theoremstyle{definition} 
\newtheorem{defn}[thm]{Definition}

\theoremstyle{remark}
\newtheorem{rem}[thm]{Remark}

\newtheorem*{ack}{Acknowledgments}

\title{Notes on direct images of pluricanonical bundles}
\author{Sho Ejiri}
\address{Department of Mathematics, Graduate School of Science, Osaka Metropolitan University, Osaka City, Osaka 558-8585, Japan}
\email{shoejiri.math@gmail.com}
\begin{document}
\maketitle
\markboth{SHO EJIRI}{Notes on direct images of pluricanonical bundles}
\begin{abstract}
In this note, we generalize slightly Popa--Schnell's theorem regarding to direct images of pluricanonical bundles to the case when the ample line bundle is not globally generated. 
We also treat the case of positive characteristic. 
\end{abstract}
\section{Introduction}
Popa and Schnell~\cite{PS14} proposed a conjecture 
that relativizes Fujita's freeness conjecture:
\begin{conj}[\textup{\cite[Conjecture~1.3]{PS14}}] \label{conj:PS}
Let the base field be an algebraically closed field of characteristic zero. 
Let $f:X\to Y$ be a morphism from a smooth projective variety $X$ 
to a smooth projective variety $Y$ of dimension $n$. 
Let $\mathcal L$ be an ample line bundle on $Y$. 
Then for every $m\ge 1$, the sheaf 
$$
f_*\omega_X^m \otimes \mathcal L^l
$$
is generated by its global sections for $l\ge m(n+1)$. 
\end{conj}
They proved the conjecture affirmatively when $|\mathcal L|$ is free: 
\begin{thm}[\textup{\cite[Theorem~1.4]{PS14}}] \label{thm:PS}
Let the base field be an algebraically closed field of characteristic zero. 
Let $f:X\to Y$ be a morphism from a smooth projective variety $X$ 
to a projective variety $Y$ of dimensions $n$. 
Let $\mathcal L$ be an ample line bundle on $Y$ with $|\mathcal L|$ free. 
Then for every $m\ge 1$, the sheaf 
$$
f_*\omega_X^m \otimes \mathcal L^l
$$
is generated by its global sections for $l\ge m(n+1)$. 
\end{thm}
Several results regarding to the above theorem are known 
(\cite{Den20, Dut20, DM19, Iwa17, Kaw02, Kol86}). 
In this note, we generalize slightly the above theorem to the case when $|\mathcal L|$ is not necessarily free. 
\begin{thm}[\textup{Theorem~\ref{thm:general}}] \label{thm:main_0}
Let the base field be an algebraically closed field of characteristic zero. 
Let $f:X\to Y$ be a morphism from a smooth projective variety $X$ 
to a projective variety $Y$ of dimension $n$. 
Let $L$ be an ample line bundle on $Y$. 
Let $j$ be the smallest positive integer such that $|\mathcal L^j|$ is free. 
Then for every $m\ge1$, the sheaf
$$
f_*\omega_X^m \otimes \mathcal L^l
$$
is generated by its global sections for $l\ge m(jn +1)$. 
\end{thm}
When $j=1$, i.e. $|\mathcal L|$ is free, 
then the above theorem is equivalent to Theorem~\ref{thm:PS}. 
When $j\ge 2$, Theorem~\ref{thm:main_0} cannot be concluded simply 
from Theorem~\ref{thm:PS}. 
Indeed, applying Theorem~\ref{thm:PS} 
in the situation of Theorem~\ref{thm:main_0}, we get that the sheaf 
$$
f_*\omega_X^m \otimes \mathcal L^{jl}
$$
is generated by its global sections for $l \ge m(n+1)$, 
but $jm(n+1)=jmn +jm > jmn + m =m(jn+1)$. 
 
We furthermore treat the case of positive characteristic. 
Note that there exist counterexamples to Conjecture~\ref{conj:PS} and Theorem~\ref{thm:PS} in positive characteristic (cf. \cite{GZZ22, SZ20}). 
The following is an analog of Theorem~\ref{thm:PS} in positive characteristic: 
\begin{thm}[\textup{\cite[Theorem~1.5]{Eji19d}}] \label{thm:PS-p}
Let the base field be an algebraically closed field of positive characteristic. 
Let $f:X\to Y$ be a surjective morphism from a smooth projective variety $X$ 
to a projective variety $Y$ of dimension $n$. 
Let $\mathcal L$ be an ample line bundle on $Y$ with $|\mathcal L|$ free. 
Suppose that $\omega_X$ is $f$-ample. 
Then there exists an integer $m_0\ge 1$ such that the sheaf
$$
f_*\omega_X^m \otimes \mathcal L^l
$$
is generated by its global sections for each $m\ge m_0$ and $l\ge m(n+1)$. 
\end{thm}
Similarly to the case of characteristic zero, 
we generalize the above theorem: 
\begin{thm}[\textup{Theorem~\ref{thm:general-p}}] \label{thm:main_p}
Let the base field be an algebraically closed field of positive characteristic. 
Let $f:X\to Y$ be a surjective morphism from a smooth projective variety $X$ 
to a projective variety $Y$ of dimension $n$. 
Let $\mathcal L$ be an ample line bundle on $Y$.  
Let $j$ be the smallest positive integer such that $|\mathcal L^j|$ is free. 
Suppose that $\omega_X$ is $f$-ample. 
Then there exists an integer $m_1\ge 1$ such that the sheaf 
$$
f_*\omega_X^m \otimes \mathcal L^l
$$ 
is generated by its global sections for each $m\ge m_1$ and $l\ge m(jn+1)$. 
\end{thm}
Even when the canonical bundle on $X$ is not $f$-ample, 
if it is ample on a fiber, 
then we can prove a similar result to the above theorem: 
\begin{thm}[\textup{Theorem~\ref{thm:general-p}}] \label{thm:main'_p}
Let the base field be an algebraically closed field of positive characteristic. 
Let $f:X\to Y$ be a surjective morphism from a smooth projective variety $X$ 
to a projective variety $Y$ of dimension $n$. 
Let $\mathcal L$ be an ample line bundle on $Y$.  
Let $j$ be the smallest positive integer such that $|\mathcal L^j|$ is free. 
Suppose that $\omega_{X_\eta}$ is ample, where $X_\eta$ is the generic fiber of $f$. 
Then there exists an integer $m_2\ge 1$ such that the sheaf 
$$ 
f_*\omega_X^m \otimes \mathcal L^l
$$ 
is generically generated by its global sections for each $m\ge m_2$ and $l\ge m(jn+1)$. 
\end{thm}
Theorems~\ref{thm:main_p} and~\ref{thm:main'_p} are follows from 
Theorem~\ref{thm:general-p} whose proof is based on 
the idea of the referee of \cite{Eji19d}. 
\begin{ack}
The author withes to express his thanks to Professors Osamu Fujino, Mihnea Popa and Christian Schnell for valuable comments. 
He is grateful to Professor Masataka Iwai for telling him about Remark~\ref{rem:smooth}. 
He also would like to thank the referee of \cite{Eji19d} for allowing to write this note using the idea of the referee of \cite{Eji19d}. 
\end{ack}
\section{Preliminaries} 
In this section, we collect several terminologies 
and definitions that are used in this note. 

Let $k$ be a field.  
We mean by a variety an integral separated scheme of finite type over $k$. 

For a normal variety $X$ 
and a $\mathbb Q$-Weil divisor $\Delta =\sum_i \delta_i\Delta_i$, 
where each $\Delta_i$ is a prime divisor on $X$, 
we denote by $\lfloor \Delta \rfloor$ (resp. $\{\Delta\}$) the Weil divisor 
$\sum_i \lfloor \delta_i \rfloor \Delta_i$ 
(resp. the effective $\mathbb Q$-Weil divisor $\Delta -\lfloor \Delta \rfloor$). 

For a variety $X$ of positive characteristic, $F_X^e:X\to X$ denotes the $e$-times iterated Frobenius morphism of $X$. 
\begin{defn} \label{defn:ggg}
Let $Y$ be a projective variety over a field $k$. 
Let $\mathcal G$ be a coherent sheaf on $Y$. 
\begin{enumerate}
\item 
Let $V$ be an open subset of $Y$. 
We say that $\mathcal G$ is \textit{generated by its global sections on $V$ $($or globally generated on $V$$)$}
if the natural morphism 
$$
H^0(Y,\mathcal G) \otimes_k \mathcal O_Y
\to 
\mathcal G
$$
is surjective on $V$. 
\item 
We say that $\mathcal G$ is \textit{generically generated by its global sections 
$($or generically globally generated$)$} if $\mathcal G$ is 
generated by its global sections on a dense open subset of $Y$.  
\end{enumerate}
\end{defn}
\begin{defn} \label{defn:F-pure}
Let $X$ be a normal variety and let $\Delta$ be an effective $\mathbb Q$-Weil divisor on $X$. 
We say that the pair $(X,\Delta)$ is $F$-pure if the composite of 
$$
\mathcal O_X 
\xrightarrow{{F_X^e}^\sharp} {F_X^e}_*\mathcal O_X
\hookrightarrow {F_X^e}_*\mathcal O_X(\lfloor (p^e-1)\Delta\rfloor)
$$
splits locally as an $\mathcal O_X$-module homomorphism for each $e\ge1$. 
\end{defn}
\section{Proofs of theorems}
In this section, we prove the main theorems. 
First, we consider the case of characteristic zero. 
\begin{thm} \label{thm:general}
Let the base field be an algebraically closed field of characteristic zero. 
Let $X$ be a normal projective variety 
and let $\Delta$ be an effective $\mathbb Q$-Weil divisor on $X$ 
such that $i(K_X+\Delta)$ is Cartier for an integer $i\ge 1$ 
and that $(X,\Delta)$ is log canonical. 
Let $Y$ be a projective variety of dimension $n$ 
and let $f:X\to Y$ be a morphism.
Let $\mathcal L$ be an ample line bundle on $Y$. 
Let $j$ be the smallest positive integer such that $|\mathcal L^j|$ is free. 
Then for each $m\ge 1$ with $i|m$, the sheaf 
$$
f_*\mathcal O_X(m(K_X+\Delta)) \otimes \mathcal L^l
$$
is generated by its global sections for each $l\ge m(jn +1)$. 
\end{thm}
\begin{proof}
The proof is almost the same as that of \cite[Theorem~1.7]{PS14}, 
and our argument follows \cite[Proof of Theorem~3.2.1]{Fuj20iitaka}. 
 
Fix $m\ge 1$ such that $m(K_X+\Delta)$ is Cartier. 
Put $\mathcal B :=\mathcal O_X(m(K_X+\Delta))$. 
We may assume that $f_*\mathcal B \ne 0$. 
Then the image of the natural morphism 
$
f^*f_*\mathcal B \to \mathcal B
$
is isomorphic to $\mathfrak b \otimes \mathcal B$ 
for an ideal sheaf $\mathfrak b$. 
Taking the log resolution of $\mathfrak b$ and $\Delta$, 
we may assume that $\mathfrak b \cong \mathcal O_X(-E)$ for 
an effective divisor $E$ such that $\mathrm{Supp}(\Delta+E)$ is simple normal crossing. 
Set $\mathcal M:=\mathcal B(-E)$. 
Then we have 
$
f^*f_*\mathcal B 
\overset{\alpha}{\twoheadrightarrow} \mathcal M 
\overset{\beta}{\hookrightarrow} \mathcal B
$ 
and 
$$
\xymatrix{
f_*\mathcal B \ar[r] \ar@/^20pt/[rrr]^{\mathrm{id}} & 
f_*f^*f_*\mathcal B \ar[r]_-{f_*\alpha} & 
f_*\mathcal M \ar@{^(->}[r]_-{f_*\beta} &
f_*\mathcal B, 
}
$$
so $f_*\mathcal M \cong f_*\mathcal B$. 
We show that there is a divisor $0\le E' \le E$ 
such that each coefficient in 
$
\Delta':=\Delta +\frac{m-1}{m}E -E'
$
is in $(0,1]$. 
We can write 
\begin{align*}
\Delta = \sum_{i=1}^\mu d_iD_i, ~d_i\in\mathbb R \cap (0,1]
\quad \textup{and} \quad
E = E_1 +\sum_{i=1}^\mu e_iD_i, ~e_i \in\mathbb Z_{\ge0}, 
\end{align*}
where $\mathrm{Supp}(E_1)$ and $\mathrm{Supp}(\Delta)$ have no common component.
Put  
$$
E':=\left\lfloor \frac{m-1}{m}E_1 \right\rfloor 
+\sum_{i=1}^\mu \left\lceil d_i + \frac{m-1}{m}e_i -1 \right\rceil D_i. 
$$
Then one can check that $E'$ is what we wanted. 
Let $s$ be the smallest integer such that 
$f_*\mathcal B \otimes \mathcal L^s$ is globally generated. 
Then, we see from the morphism 
$$
f^*f_*\mathcal B \otimes f^* \mathcal L^{s}
\twoheadrightarrow \mathcal M \otimes f^*\mathcal L^{s}
$$
that 
$
\mathcal M \otimes f^*\mathcal L^{s}
$
is globally generated. 
Then there is a smooth member $D\in |\mathcal M \otimes f^*\mathcal L^{s}|$ 
such that $\mathrm{Supp}(D)$ and $\mathrm{Supp}(\Delta')$ 
have no common components and 
$\mathrm{Supp}(\Delta' +D)$ is simple normal crossing. 
Let $L$ be a Cartier divisor on $Y$ such that $\mathcal O_Y(L)\cong\mathcal L$. 
We now have 
$$
m(K_X+\Delta) +s f^*L \sim D +E, 
$$
from which we obtain 
$$
(m-1)(K_X+\Delta) 
\sim_{\mathbb Q} \frac{m-1}{m}D +\frac{m-1}{m}E -\frac{(m-1)}{m}sf^*L, 
$$
and so we get 
\begin{align*}
& m(K_X+\Delta) -E' +lf^*L 
\\ \sim_{\mathbb Q} & K_X +\Delta -E' +\frac{m-1}{m}D +\frac{m-1}{m}E
+\left( l -\frac{m-1}{m} s \right) f^*L 
\\ = & K_X +\Delta' +\frac{m-1}{m}D 
+\left( l -\frac{m-1}{m} s \right) f^*L 
\end{align*}
for an $l \ge 1$. 
Then it follows from the projection formula and the vanishing theorem 
(\cite[Theorem~3.2]{Amb03} or \cite[Theorem~6.3]{Fuj11}), 
that 
$$
H^i\left(Y, f_*\mathcal B \otimes \mathcal L^l \right) =0
$$
for each $i>0$ and $l >\frac{m-1}{m}s$. 
Note that since $0 \le E' \le E$ we have 
$$
f_*\mathcal M
\overset{\cong}{\hookrightarrow}
f_*\mathcal O_X( m(K_X+\Delta) -E' )
\overset{\cong}{\hookrightarrow}
f_*\mathcal B. 
$$
Hence, if $l>\frac{m-1}{m}s +jn$, then 
$
f_*\mathcal B \otimes \mathcal L^l
$
is $0$-regular with respect to $\mathcal L^j$ 
in the sense of Castelnuovo--Mumford, 
and so it is globally generated (\cite[Theorem~1.8.3]{Laz04I}). 
By the definition of $s$, we obtain
$$
s \le \frac{m-1}{m} s +jn +1, 
$$
and thus $s \le m(jn +1)$. 
Therefore, if $l\ge m(jn +1)$, then 
$$
\frac{m-1}{m}s +jn +1 
\le \frac{m-1}{m}m(jn+1) +jn +1 
= m(jn+1) 
\le l, 
$$
so $f_*\mathcal B \otimes \mathcal L^l$ is globally generated, 
which is our claim. 
\end{proof}
\begin{rem} \label{rem:smooth}
The author learned from Professor Masataka Iwai that 
if $X$ and $Y$ are smooth and if $f:X\to Y$ is smooth, 
then for every ample line bundle $\mathcal L$ on $Y$,
$$
f_*\omega_X^m \otimes \mathcal L^l
$$
is generated by its global sections for each $l\ge m(n+1) +n(j-1)$, 
where $j$ is the smallest positive integer such that $|\mathcal L^j|$ is free. 
The proof is based on an analytic method that is similar to that of \cite{Iwa17}, 
which uses a singular Hermitian metric constructed in \cite[Section 11]{AS95}. 
\end{rem}
Next, we treat the case of positive characteristic. 
\begin{thm} \label{thm:general-p}
Let the base field be an algebraically closed field of characteristic $p>0$. 
Let $X$ be a normal projective variety and let $\Delta$ be an effective $\mathbb Q$-Weil divisor on $X$ such that $i(K_X+\Delta)$ is Cartier for an integer $i>0$ not divisible by $p$. 
Let $Y$ be a projective variety of dimension $n$ 
and let $f:X\to Y$ be a morphism.
Suppose that there exists a dense open subset $V\subseteq Y$ such that 
\begin{itemize}
\item $(U,\Delta|_U)$ is $F$-pure, where $U:=f^{-1}(V)$, and 
\item $K_U+\Delta|_U$ is ample over $V$. 
\end{itemize}
Let $\mathcal L$ be an ample line bundle on $Y$ 
and let $j$ be the smallest positive integer such that $|\mathcal L^j|$ is free. 
Then there exists an integer $m_0\ge 1$ such that the sheaf 
$$
f_*\mathcal O_X(m(K_X+\Delta)) \otimes \mathcal L^l
$$
is generated by its global sections for each $m \ge m_0$ with $i|m$ 
and each $l\ge m(jn +1)$. 
\end{thm}
To prove the theorem, we need the following lemma:
\begin{lem} \label{lem:glgen}
Let the base field be a field of characteristic $p>0$. 
Let $Y$ be a projective variety of dimension $n$ 
and let $\mathcal E$ be a coherent sheaf on $Y$. 
Let $L$ be an ample Cartier divisor on $Y$ 
and let $j$ be a positive integer such that $|jL|$ is free. 
Let $\{a_e\}_{e\ge1}$ be a sequence of positive integers such that 
$a_e/p^e$ converses to $\varepsilon +jn$ 
for an $\varepsilon \in \mathbb R_{>0}$. 
Then there exists an $e_0\ge 1$ such that 
$$
{F_Y^e}_*(\mathcal E(a_eL))
$$
is generated by its global sections for each $e\ge e_0$. 
\end{lem}
\begin{proof}
We prove that ${F_Y^e}_*(\mathcal E(a_eL))$ is $0$-regular 
with respect to $jL$ in the sense of Castelnuovo--Mumford. 
If this holds, then the assertion follows from \cite[Theorem~1.8.3]{Laz04I}. 
For $0<i \le n$, we have 
\begin{align*}
H^i\big(Y, \big({F_Y^e}_*(\mathcal E(a_eL))\big)(-ijL) \big)
& \cong H^i\big(Y, {F_Y^e}_*\big(\mathcal E((a_e-p^eij)L)\big) \big)
\\ & \cong H^i\big(Y, \mathcal E((a_e-p^eij)L) \big) =:V_{e,i} 
\end{align*}
by the projection formula. 
Since $L$ is ample, there is an $m_0\ge 1$ such that 
$$
H^i(Y, \mathcal E(mL)) =0
$$ 
for each $m\ge m_0$ and $i>0$. 
By the assumption, there is an $e_0\ge 1$ such that 
$a_e-p^eij\ge m_0$ for each $e\ge e_0$, 
so $V_{e,i}=0$ for each $e\ge e_0$. 
\end{proof} 
\begin{proof}[Proof of Theorem~\ref{thm:general-p}]
The proof is based on the idea of the referee of \cite{Eji19d}. 
Let $e_0\ge 1$ be an integer such that $i|(p^{e_0}-1)$. 
For $m\ge1$ with $i|m$ and $e\ge 1$ with $e_0|e$, 
applying $\mathcal Hom_{\mathcal O_X}(?, m(K_X+\Delta))$ to the composite of 
\begin{align} \label{mor:1}
\mathcal O_X 
\xrightarrow{{F_X^e}^\sharp} {F_X^e}_*\mathcal O_X
\hookrightarrow {F_X^e}_*\mathcal O_X((p^e-1)\Delta), 
\end{align}
we obtain the morphism 
\begin{align} \label{mor:2}
{F_X^e}_*\mathcal O_X\big(((m-1)p^e+1)(K_X+\Delta) \big) 
\to \mathcal O_X(m(K_X+\Delta))
\end{align}
by the Grothendieck duality. 
This is surjective on $U$ by the $F$-purity of $(U,\Delta|_U)$. 
Set $\mathcal G(m) :=f_*\mathcal O_X(m(K_X+\Delta))$ 
for each $m\ge 1$ with $i|m$. 
Pushing morphism~(\ref{mor:2}) forward by $f$, we get 
\begin{align} \label{mor:3}
{F_Y^e}_*\mathcal G((m-1)p^e+1)
\to \mathcal G(m). 
\end{align}
There is an $m_0\ge 1$ such that the above morphism is surjective on $V$ 
for each $m\ge m_0$ with $i|m$, since $(K_X+\Delta)|_U$ is ample over $V$ 
(this is proved by the same argument as that of \cite[Example~3.11]{Eji17}).
Also, since $(K_X+\Delta)|_U$ is ample over $V$, 
there is a $d\ge m_0$ with $i|d$ and $n_0\ge 1$ such that the natural morphism
\begin{align} \label{mor:4}
\mathcal G(d) \otimes \mathcal G(n)
\to \mathcal G(d+n)
\end{align}
is surjective on $V$ for each $n\ge n_0$ with $i|n$. 
For $e, m\ge 1$, let $q_{e,m}$ and $r_{e,m}$ be integers such that 
$(m-1)p^e+1 = q_{e,m} d +r_{e,m}$ and $n_0 \le r_{e,m} <n_0+d$. 
Note that if $e_0|e$ and $i|m$ then $i|r_{e,m}$. 
Then, for each $m\ge1$ with $i|m$ and $e\ge 1$ with $e_0|e$, 
the natural morphism 
\begin{align} \label{mor:5}
S^{q_{e,m}}(\mathcal G(d)) \otimes \mathcal G(r_{e,m}) 
\to \mathcal G((m-1)p^e+1)
\end{align}
is surjective on $V$, 
where $S^m(?)$ denotes $m$-th symmetric product. 
Combining morphisms~(\ref{mor:3}) and~(\ref{mor:5}), 
we obtain the morphism 
\begin{align} \label{mor:6}
{F_Y^e}_* \big( S^{q_{e,m}}(\mathcal G(d)) \otimes \mathcal G(r_{e,m})\big)
\to \mathcal G(m)
\end{align}
for each $m\ge m_0$ with $i|m$ and $e\ge 1$ with $e_0|e$.
This is surjective on $V$ by the construction. 
Put $\mathcal E:=\bigoplus_{n_0 \le r <n_0 +d,~i|r} \mathcal G(r)$. 
By morphism~(\ref{mor:6}), we have the morphism 
\begin{align} \label{mor:7}
{F_Y^e}_* \big( S^{q_{e,m}}(\mathcal G(d)) \otimes \mathcal E\big)
\to \mathcal G(m)
\end{align}
for each $m\ge m_0$ with $i|m$ and $e\ge 1$ with $e_0|e$.
For each $m\ge 1$ with $i|m$, let $s_m$ be the smallest integer 
such that $\mathcal G(m) \otimes \mathcal L^{s_m}$ is globally generated. 
Taking the tensor product of morphism~(\ref{mor:7}) 
and $\mathcal O_X(lL)$, 
where $l\ge 1$ and $L$ is a Cartier divisor on $Y$ 
with $\mathcal O_Y(L)\cong \mathcal L$, 
we get the sequence of morphisms
\begin{align}
\mathcal G(m)(lL) 
& \leftarrow {F_Y^e}_*\big( 
S^{q_{e,m}}(\mathcal G(d)) \otimes \mathcal E (lp^eL) 
\big) \label{mor:8}
\\ & \cong {F_Y^e}_*\big( 
S^{q_{e,m}}(\mathcal G(d))(q_{e,m}s_d L)
\otimes \mathcal E((lp^e -q_{e,m}s_d)L) 
\big) \label{mor:9}
\\ & \leftarrow {F_Y^e}_*\left( 
\left(\bigoplus \mathcal O_Y\right) 
\otimes \mathcal E(lp^e -q_{e,m}s_d)L)
\right) \label{mor:10}
\\ & \cong \bigoplus {F_Y^e}_*\big(\mathcal E( (lp^e-q_{e,m}s_d)L) \big), 
\label{mor:11}
\end{align}
where (\ref{mor:10}) follows from the fact that 
$
S^{q_{e,m}}(\mathcal G(d))(q_{e,m}s_dL)
$
is globally generated on $V$. 
The above morphisms are surjective on $V$.
Since 
$$
\frac{lp^e -q_{e,m}s_d}{p^e} 
\xrightarrow{e\to +\infty} l -\frac{m-1}{d}s_d, 
$$
if $l>\frac{m-1}{d}s_d +jn$ and $e\gg0$, then sheaf (\ref{mor:11}) is 
globally generated by Lemma~\ref{lem:glgen}, 
and hence $\mathcal G(m)(lL)$ is globally generated on $V$. 
By the definition of $s_m$, we have 
$$
s_m \le \frac{m-1}{d}s_d +jn +1. 
$$
When $m=d$, we get $s_d \le d(jn +1)$. 
Therefore, when $m$ is not necessarily equal to $d$, if $l \ge m(jn+1)$, then 
$$
\frac{m-1}{d}s_d +jn +1 
\le \frac{m-1}{d}d(jn +1) +jn +1
=m(jn+1)
\le l, 
$$
so $\mathcal G(m)(lL)$ is globally generated on $V$, 
which is our claim. 
\end{proof}
\bibliographystyle{abbrv}
\bibliography{ref.bib}
\end{document}